\pdfoutput=1
\documentclass[a4paper,10pt]{article}

\usepackage[english]{babel}
\usepackage[T1]{fontenc}
\usepackage{amsmath,amssymb,amsthm}
\usepackage{mathtools}
\newtheorem{theorem}{Proposition}[section]
\newtheorem{prop}[theorem]{Proposition}
\newtheorem{lemma}[theorem]{Lemma}
\usepackage{graphicx}
\usepackage{hyperref}
\hypersetup{
    colorlinks=true,
    citecolor=black,
    linkcolor=blue,
    urlcolor=blue,
}
\usepackage[capitalise]{cleveref}
\crefname{equation}{}{}
\crefname{prop}{Prop.}{Props.}
\Crefname{prop}{Prop.}{Props.}
\crefname{lemma}{Lemma}{Lemmas}
\Crefname{lemma}{Lemma}{Lemmas}
\crefname{algorithm}{Alg.}{Alg.}

\newtheorem{algorithm}{Algorithm}
\usepackage{tikz}
\usetikzlibrary{fit,positioning}
\usetikzlibrary{arrows}
\usetikzlibrary{arrows.meta,quotes}
\usepackage{pgfplots}
\usepgfplotslibrary{groupplots}
\pgfplotsset{compat=1.16}
\usepackage{algorithm,algorithmicx,algpseudocode}

\usepackage[numbers]{natbib}
\usepackage{authblk}

\usepackage{xspace}
\usepackage{amsfonts}

\newcommand{\mat}[1]{#1}
\newcommand{\vek}[1]{#1}
\newcommand{\expect}[1]{\mathbb{E}\left[#1 \right]}

\newcommand{\reals}{\mathbb{R}}
\newcommand{\diag}{\mathrm{diag}}
\newcommand{\grad}{\nabla}
\newcommand{\rad}{\text{rad}\xspace}
\newcommand{\argMin}[1]{\underset{#1}{\mathrm{arg\ min}\,}}

\newcommand{\norm}[1]{\left\lVert #1 \right\lVert}
\newcommand{\twoNorm}[1]{\norm{#1}_2}
\renewcommand{\d}{\mathrm{d}}

\newcommand{\numMcRuns}{100\xspace}

\newcommand{\rmse}{RMSE\xspace}
\newcommand{\nees}{NEES\xspace}
\newcommand{\kf}{KF\xspace}

\newcommand{\iekf}{IEKF\xspace}
\newcommand{\rts}{RTS\xspace}

\newcommand{\eks}{EKS\xspace}
\newcommand{\ieks}{IEKS\xspace}
\newcommand{\slr}{\text{SLR}\xspace}

\newcommand{\iplf}{IPLF\xspace}

\newcommand{\ipls}{IPLS\xspace}
\newcommand{\gn}{GN\xspace}
\newcommand{\lm}{LM\xspace}

\newcommand{\lmIeks}{LM--IEKS\xspace}
\newcommand{\diplf}{D\iplf}
\newcommand{\lmIpls}{LM--IPLS\xspace}

\newcommand{\normal}{\mathcal{N}}

\newcommand{\err}[1]{\vek{#1}}
\newcommand{\covMat}[1]{\mat{#1}}
\newcommand{\ts}{k}
\newcommand{\tsFin}{K}
\newcommand{\x}{\vek{x}}
\newcommand{\xHat}{\hat{\x}}
\newcommand{\dimState}{d_{\x}}
\newcommand{\procFn}{\vek{f}}
\newcommand{\procNoise}{\err{q}}
\newcommand{\procCov}{\covMat{Q}}

\newcommand{\meas}{\vek{y}}
\newcommand{\dimMeas}{d_{\meas}}
\newcommand{\measFn}{{h}}
\newcommand{\measNoise}{\err{r}}
\newcommand{\measCov}{\covMat{R}}
\newcommand{\measSeq}{\meas_{1:\tsFin}}
\newcommand{\allDim}{D}

\newcommand{\pred}[1]{#1 \mid #1-1}

\newcommand{\upd}[1]{#1 \mid #1}
\newcommand{\smoothCurr}[1]{#1 \mid \tsFin}

\newcommand{\priorMean}{\hat{\x}_{1 \mid 0}}
\newcommand{\xPred}[1]{\xHat_{\pred{#1}}}

\newcommand{\xUpd}[1]{\xHat_{\upd{#1}}}

\newcommand{\xSmooth}[1]{\xHat_{\smoothCurr{#1}}}

\renewcommand{\P}{\mat{P}}
\newcommand{\priorCov}{\hat{\P}_{1 \mid 0}}
\newcommand{\PHat}{\hat{\P}}
\newcommand{\PPred}[1]{\PHat_{\pred{#1}}}

\newcommand{\PUpd}[1]{\PHat_{\upd{#1}}}
\newcommand{\PSmooth}[1]{\PHat_{\smoothCurr{#1}}}

\newcommand{\procA}{\mat{F}}
\newcommand{\procAIt}[1]{\iterQty{\mat{F}}{#1}}
\newcommand{\procB}{\vek{b}}
\newcommand{\procLinErr}{\err{\omega}}
\newcommand{\procLinCov}{\covMat{\Omega}}
\newcommand{\procLinCovIt}[1]{\iterQty{\procLinCov}{#1}}
\newcommand{\measH}{\mat{H}}
\newcommand{\measHIt}[1]{\iterQty{\mat{H}}{#1}}
\newcommand{\measC}{\vek{c}}
\newcommand{\measLinErr}{\err{\gamma}}
\newcommand{\measLinCov}{\covMat{\Gamma}}
\newcommand{\measLinCovIt}[1]{\iterQty{\measLinCov}{#1}}
\newcommand{\Q}{\mat{Q}}

\newcommand{\pdf}{PDF\xspace}

\newcommand{\innCov}{\mat{\Sigma}}
\newcommand{\R}{\mat{R}}

\newcommand{\K}{\mat{K}}

\newcommand{\iterQty}[2]{#1^{(#2)}}
\newcommand{\xHatIt}[1]{\iterQty{\xHat}{#1}}
\newcommand{\xSeq}{\x_{1:\tsFin}}
\newcommand{\xHatItSeq}[1]{\iterQty{\xHat_{1:\tsFin}}{#1}}
\newcommand{\PHatItSeq}[1]{\iterQty{\PHat_{1:\tsFin}}{#1}}
\newcommand{\PHatIt}[1]{\iterQty{\PHat}{#1}}
\newcommand{\PHatItInv}[1]{\left(\PHat^{(#1)}_{\ts}\right)^{-1}}

\newcommand{\costFn}{L}

\newcommand{\stateBar}{\bar{\x}}
\newcommand{\measBar}{\bar{\vek{\meas}}}

\newcommand{\PsiSlr}[1]{\mat{\Psi}_{#1}}
\newcommand{\PhiSlr}[1]{\mat{\Phi}_{#1}}

\newcommand{\ThetaSeq}{\Theta_{1:\tsFin}}

\newcommand{\xHatSeqSmooth}[1]{\xHat_{1:\tsFin}^{#1}}
\newcommand{\PHatSeqSmooth}[1]{\PHat_{1:\tsFin}^{#1}}
\newcommand{\affAppParam}[1]{\iterQty{\ThetaSeq}{#1}}

\newcommand{\gnObj}{\costFn_{\mathrm{\gn}}}
\newcommand{\gnObjLin}{\tilde{\costFn}_{\mathrm{\gn}}^{(i)}}
\newcommand{\gnR}{\rho}
\newcommand{\gnRIt}[1]{\iterQty{\gnR}{#1}}
\newcommand{\gnRLin}{\tilde{\gnR}^{(i)}}
\newcommand{\augObs}[1]{\vek{z}_{#1}}
\newcommand{\augObsApprox}[1]{\tilde{\vek{z}}_{#1}^{(i)}}
\newcommand{\augQ}{\mat{\Sigma}}
\newcommand{\augQInv}{\augQ^{-1}}
\newcommand{\augQInvSqrt}{\augQ^{-1/2}}
\newcommand{\augQInvSqrtT}{\augQ^{-T/2}}

\newcommand{\lsIeks}{LS--\ieks}
\newcommand{\lsIpls}{LS--\ipls}
\newcommand{\searchDir}[1]{\Delta \xHatIt{#1}}
\newcommand{\searchDirSeq}[1]{\Delta \xHatItSeq{#1}}
\newcommand{\lsPoint}{\xHat(\alpha)}

\newcommand{\lmRegParamIt}[1]{\iterQty{\lambda}{#1}}
\newcommand{\lmRegParamInit}{\iterQty{\lambda}{0}}
\newcommand{\regMat}{\mat{S}}
\newcommand{\lmObj}[1]{\iterQty{\costFn_{\mathrm{\lm}}}{#1}}
\newcommand{\lmObjLin}{\tilde{\costFn}_{\mathrm{\lm}}^{(i)}}
\newcommand{\gnRLM}[1]{\iterQty{\gnR_{\mathrm{\lm}}}{#1}}

\newcommand{\lmRegSqrtT}{\left[ (\lmRegParamIt{i})^{-1} \regMat^{(i)} \right]^{-T/2}}
\newcommand{\lmErr}{\err{e}}

\newcommand{\gnIeksObj}{\costFn_{\mathrm{\ieks}} }

\newcommand{\iplsLabel}{}
\newcommand{\augObsIpls}{\augObs{\iplsLabel}}
\newcommand{\augObsIplsApprox}{\augObsApprox{\iplsLabel}}
\newcommand{\augQIplsInv}{\augQInv_{\augObsIpls}}
\newcommand{\augQIplsInvSqrt}{\augQInvSqrt_{\augObsIpls}}
\newcommand{\augQIplsInvSqrtT}{\augQInvSqrtT_{\augObsIpls}}
\newcommand{\gnIplsObj}[1]{\iterQty{\costFn_{\mathrm{\ipls}}}{#1}}
\newcommand{\gnIplsObjLin}[1]{\iterQty{\tilde{\costFn}_{\mathrm{\ipls}}}{#1}}

\pgfplotscreateplotcyclelist{onlyColor}{
  {blue},
  {red},
  {green},
  {violet},
  {orange},
  {cyan},
}
\pgfplotscreateplotcyclelist{colorAndLine}{
  {blue},
  {red,dashed},
  {green,dotted},
  {violet},
  {orange,densely dashed},
  {cyan,densely dotted},
}
\definecolor{color0}{rgb}{0.12156862745098,0.466666666666667,0.705882352941177}
\definecolor{color1}{rgb}{1,0.498039215686275,0.0549019607843137}
\definecolor{color2}{rgb}{0.172549019607843,0.627450980392157,0.172549019607843}
\definecolor{color3}{rgb}{0.83921568627451,0.152941176470588,0.156862745098039}
\definecolor{color4}{rgb}{0.15,0.8,0.15}
\definecolor{color5}{rgb}{0.15,0.15,0.15}

\title{Levenberg--Marquardt and Line-Search Iterated Posterior Linearisation Smoothing}

\author[1]{Jakob Lindqvist}
\author[2]{Simo S\"{a}rkk\"{a}}
\author[3]{\'{A}ngel~F.~Garc{\'i}a-Fern{\'a}ndez}
\author[4]{Matti Raitoharju}
\author[1]{Lennart Svensson}

\affil[1]{Dept. of Electrical Engineering, Chalmers University of Technology, Gothenburg, Sweden}
\affil[2]{Dept. of Electrical Eng. and Automation, Aalto University, Esbo, Finland}
\affil[3]{ETSI de Telecomunicaci\'on, Universidad Polit\'ecnica de Madrid, Madrid, Spain}
\affil[4]{Faculty of Information Technology and Communication Sciences, Tampere University, Tampere, Finland}

\date{} 

\begin{document}
\maketitle

\begin{abstract}
	This paper considers the problem of iterative Bayesian smoothing in nonlinear state-space models with additive noise using Gaussian approximations.
	Iterative methods are known to improve smoothed estimates but are not guaranteed to converge, motivating the development of methods with better convergence properties.
	The aim of this article is to extend Levenberg--Marquardt (LM) and line-search versions of the classical iterated extended Kalman smoother (IEKS) to the iterated posterior linearisation smoother (IPLS).
	The IEKS has previously been shown to be equivalent to the Gauss--Newton (GN) method.
	We derive a similar GN interpretation for the IPLS and use this to develop extensions to the IPLS, with improved convergence properties.
	We show that an LM extension for the IPLS can be achieved with a simple modification of the smoothing iterations, enabling algorithms with efficient implementations.
	We also derive the Armijo--Wolfe step length conditions for the IPLS enabling an efficient inexact line-search method.
	Our numerical experiments show the benefits of these extensions in highly nonlinear scenarios.
\end{abstract}

\section{Introduction}
Smoothing is a form of state estimation, where past states of a stochastically evolving process are estimated from a noisy measurement sequence.
It has wide-ranging applications in navigation, target tracking and communications \citep{app:nav_and_track,app:audio}.
From a Bayesian perspective, the standard objective is to obtain the posterior probability density function (\pdf) of past states given all the measurements, called fixed-interval smoothing,
or the marginal \pdf for a specific state, given all measurements, called fixed-point smoothing \citep{bayes_filter_book,opt_state_est_book}.

General Gaussian Rauch--Tung--Striebel (\rts) smoothers form a family of methods which utilises the problem structure to efficiently calculate a Gaussian marginal posterior distribution over the states.
The name stems from the \rts smoother which, for linear/affine and Gaussian systems, computes the exact smoothing distributions \citep{og_rts,bayes_filter_book}.
For non-linear systems, general Gaussian \rts smoothers use a linear approximation for the nonlinearities including the covariance matrix of the linearisation error,
to then apply the closed-form \rts smoothing on the approximated system \citep{smoothing_non_lin_state_space_models}.
The choice of linearisation method, including the covariance matrix of the linearisation error, defines different members of this smoother family and the quality of the smoothing approximation.
Examples include first-order Taylor expansion in the \textit{Extended Kalman Smoother} \eks \citep{bayes_filter_book}
and \textit{statistical linear regression} (\slr) \citep{slr_results} with sigma point methods for the \textit{Unscented \rts smoother} \citep{ukf,urtss} and the \textit{Cubature \rts smoother} \citep{crtss}.

The point around which linearisation is done can greatly impact the resulting state estimates.
In general, we prefer to perform linearisation around points close to the posterior estimates, motivating iterative extensions of the aforementioned smoothers \citep{pls}.
A full smoothed trajectory is repeatedly computed, with linearisations done around the most recent estimates.
The iterative refinement can lead to significantly improved performance.
Examples of iterative methods are the \textit{Iterated} \eks (\ieks) and the \textit{Iterated posterior linearisation smoother} (\ipls) \citep{app:nav_and_track,pls}.

To improve the convergence properties of iterative methods, damping or regularisation can be used.
A common approach is to relate the smoothing procedure to an optimisation method, such as Gauss--Newton (\gn) \citep{bellaire,damped_plf,bell_iekf,bell_ieks,map_est,ekf_mods_opt}.
Regularised versions can then be created to guarantee a non-increasing cost function, such as the Levenberg--Marquardt (\lm) method, an extension of the \gn method \citep{num_opt}.
Existing works have investigated iterative extensions, in particular for the related filtering problem.
A more numerically stable update step for the \iekf, as well as the use of \lm regularisation was proposed in \citep{bellaire}.
The \textit{Damped \iplf (\diplf)} \citep{damped_plf} also used a line-search \gn algorithm to improve the \textit{Iterated posterior linearisation filter} (\iplf).
Similar \lm extensions as explored in this paper were also developed in \citep{lm_ieks_1,lm_ieks_2,lm_observation} for analytically linearised smoothers.

\section{Problem formulation}
\noindent In smoothing, we consider the problem of estimating a sequence of latent states in a Markov process
$\xSeq \coloneqq \left(\x_{1}, \x_{2}, \dots, \x_{\tsFin} \right)$, where $\x_{\ts} \in \reals^{\dimState},\ \ts = 1, \dots, \tsFin$,
based on noisy measurements
$\measSeq \coloneqq \left(\meas_{1}, \meas_{2}, \dots, \meas_{\tsFin} \right)$, where $\meas_{\ts} \in \reals^{\dimMeas},\ \ts = 1, \dots, \tsFin$, and $\tsFin$ is the final time step.

We approximate the distribution of the states as Gaussian, parameterised by their state means and covariances for every time step $\ts = 1, \dots, \tsFin$.
We use the notation $\xHat_{\ts \mid \ts'}, \PHat_{\ts \mid \ts'}$ for the state mean and covariance estimates at time step $\ts$ based on the measurements up to and including time step $\ts'$.
For smoothing, the aim is to estimate the \textit{smoothed} mean $\xSmooth{\ts}$ and covariance $\PSmooth{\ts}$, for all timesteps $\ts$, given all measurements $\measSeq$.

The state-space model is specified by its \textit{motion} and \textit{measurement} models
\begin{align}
  \label{eq:def:state_space_model}
  \x_{\ts+1} &= \procFn_{\ts}(\x_{\ts}) + \procNoise_{\ts},
  &\procNoise_{\ts} \sim \normal(\vek{0}, \procCov_{\ts}), \nonumber \\
  \meas_{\ts} &= \measFn_{\ts}(\x_{\ts}) + \measNoise_{\ts},
  &\measNoise_{\ts} \sim \normal(\vek{0}, \measCov_{\ts}),
\end{align}
where $\procFn_{\ts}: \reals^{\dimState} \to \reals^{\dimState}, \procCov_{\ts} \in \reals^{\dimState \times \dimState}, \measFn_{\ts}: \reals^{\dimState} \to \reals^{\dimMeas}$ and $\measCov_{\ts} \in \reals^{\dimMeas \times \dimMeas}$ are assumed to be known
and the process and measurement noises, $\procNoise_{\ts} \in \reals^{\dimState}$ and $\measNoise_{\ts} \in \reals^{\dimMeas}$, are assumed to be independent.
The initial prior $\x_{1} \sim \normal(\priorMean, \priorCov)$ is assumed to be known.

An important special case of space-space models is the linear (affine) Gaussian model.
In this paper, we approximate motion and measurement models on the form:
\begin{align}
  \label{eq:def:lin:state_space_model}
  \x_{\ts+1} &= \procA_{\ts} \x_{\ts} + \procB_{\ts} + \procLinErr_{\ts} + \procNoise_{\ts},
    &\procLinErr_{\ts} \sim \normal(\vek{0}, \procLinCov_{\ts}), \nonumber \\
  \meas_{\ts} &= \measH_{\ts} \x_{\ts} + \measC_{\ts} + \measLinErr_{\ts} + \measNoise_{\ts},
    &\measLinErr_{\ts} \sim \normal(\vek{0}, \measLinCov_{\ts}),
\end{align}
where for all time steps,
$\procA_{\ts} \in \reals^{\dimState \times \dimState}, \procB_{\ts} \in \reals^{\dimState}$ and
$\measH_{\ts} \in \reals^{\dimMeas \times \dimState}, \measC_{\ts} \in \reals^{\dimMeas}$ constitute the linear mappings and
the noise processes $\procNoise_\ts$, $\procLinErr_\ts$, $\measNoise_\ts$, and $\measLinErr_\ts$ are white noises, mutually independent, and independent of the initial state $\x_1$ \cite{opt_state_est_book,opt_filter_book}.
For such systems, the linear \rts smoother computes the exact marginal posterior \pdf in closed-form \cite{bayes_filter_book}.

General Gaussian \rts smoothers can handle non-linear/non-Gaussian systems.
These smoothers perform two steps to obtain a tractable smoothing algorithm \cite{bayes_filter_book,smoothing_non_lin_state_space_models}.
First, they approximate the original models in \cref{eq:def:state_space_model} as a linear Gaussian model of the form in \cref{eq:def:lin:state_space_model}.
Second, they compute the posterior distributions for this approximate model exactly using the \rts smoother.
The family of general Gaussian \rts smoothers includes the \eks, the unscented \rts smoother and the cubature \rts smoother, and the algorithms in this family only differ in how
the linearisation parameters $\Theta_{1:\tsFin} = (\procA_{\ts}, \procB_{\ts}, \procLinCov_{\ts}, \measH_{\ts}, \measC_{\ts}, \measLinCov_{\ts}), \ts=1, \dots, \tsFin$, are chosen.

Iterative extensions of general Gaussian \rts smoothers perform repeated steps of linearisation and \rts smoothing.
The sequence of estimates produced by iterative smoothers is denoted $\xHatItSeq{i}, \PHatItSeq{i}$,
where the superscript $(i)$ indicates that these are the smoothed estimates for iteration $i$.
A superscript without parentheses refers to a special iterate, labelled by the superscript, for instance, a fixed point $\xHatSeqSmooth{*}, \PHatSeqSmooth{*}$.
The starting points are the initial estimates $\xHatItSeq{0}, \PHatItSeq{0}$, which form the basis for selecting the initial linearisation parameters $\affAppParam{0}$.
We iteratively find new estimates $\xHatItSeq{i}, \PHatItSeq{i}$ with closed form \rts smoothing, and use them to select new linearisation parameters $\affAppParam{i+1}$.
Ideally, the iterative process is repeated until convergence.
We say that the process does not converge if it diverges or if the resulting estimates explain data poorly or are unreasonable.

In this paper, we propose two versions of such iterated smoothers that converge for a large set of initial estimates and measurement realisations.

\section{Background}
\subsection{Linearisations used in smoothing}
\noindent All general Gaussian \rts smoothers use the linear \rts smoother.
What sets them apart is the different methods of linearisation.

The \eks is a well-known general Gaussian \rts smoother \cite{bayes_filter_book}.
It makes the linear approximation through an analytical linearisation.
For instance, to linearise the motion model in \cref{eq:def:state_space_model}, at time step $\ts$, it uses
\begin{subequations}
  \begin{align}
    \label{eq:ext:jacobian}
    \procA_{\ts}(\xHat_{\ts}) &= J_{\procFn_{\ts}}(\xHat_{\ts}), \\
    \label{eq:ext:offset}
    \procB_{\ts}(\xHat_{\ts}) &= \procFn_{\ts}(\xHat_{\ts}) - \procA_{\ts}(\xHat_{\ts}) \xHat_{\ts}, \\
    \label{eq:ext:cov}
    \procLinCov_{\ts}(\xHat_{\ts}) &= \vek{0},
  \end{align}
\end{subequations}
where $J_{\procFn_\ts}(\xHat)$ is the Jacobian of $\procFn_\ts$, evaluated at $\xHat$.
Linearisation is done at some estimate of the mean of the state $\xHat_{\ts}$.
For instance, the ordinary \eks uses the updated means $\xUpd{\ts}$.

Another category of smoothers uses \slr to form the linear approximation,
such as the unscented \rts and cubature \rts smoothers.
The \slr for $\procFn_{\ts}$ with respect to the density $p(\x_{\ts})$, with mean $\xHat_{\ts}$ and covariance $\PHat_{\ts}$, provides the parameters \cite{slr_results}
\begin{subequations}
  \begin{align}
    \label{eq:slr:jacobian}
    \procA_{\ts}(\xHat_{\ts}, \PHat_{\ts}) &= \PsiSlr{\procFn_{\ts}}^{\top} \PHat_{\ts}^{-1}, \\
    \label{eq:slr:offset}
    \procB_{\ts}(\xHat_{\ts}, \PHat_{\ts}) &= \stateBar_{\ts} - \procA_{\ts} \xHat_{\ts}, \\
    \label{eq:slr:cov}
    \procLinCov_{\ts}(\xHat_{\ts}, \PHat_{\ts}) &= \PhiSlr{\procFn_{\ts}} - \procA_{\ts}  \PHat_{\ts} \procA_{\ts}^{\top},
  \end{align}
\end{subequations}
where
\begin{align}
  \label{eq:slr:moments}
  \stateBar_{\ts} &=
  \int \procFn_{\ts}(\x_{\ts}) p(\x_{\ts}) \d\x_{\ts}, \nonumber \\
  \PsiSlr{\procFn_{\ts}}  &=
  \int (\x_{\ts} - \xHat_{\ts})(\procFn_{\ts}(\x_{\ts}) - \stateBar_{\ts})^{\top} p(\x_{\ts}) \d\x_{\ts}, \nonumber \\
  \PhiSlr{\procFn_{\ts}} &=
  \int (\procFn_{\ts}(\x_{\ts}) - \stateBar_{\ts})(\procFn_{\ts}(\x_{\ts}) - \stateBar_{\ts})^{\top} p(\x_{\ts}) \d\x_{\ts}.
\end{align}
When the linearisation is done with respect to \(p(\x_{\ts}) = \normal(\x_{\ts}; \xHat_{\ts}, \PHat_{\ts})\),
the \slr approximation depends on both the mean $\xHat_{\ts}$ and covariance $\PHat_{\ts}$.
The moments in \cref{eq:slr:moments} are not tractable for a general function $\procFn_{\ts}$
and some form of approximation is needed,
such as a Monte Carlo method or, more commonly, a sigma point method \cite{seq_monte_carlo,bayes_filter_book}.

For non-iterative methods such as the \eks, the linearisation is done at the predicted and filtered estimates for both the filtering and smoothing passes.
With non-linear motion or measurement models, the risk is that the linearisation is a poor fit for the models over the region of interest.
Furthermore, the choice of linearisation point based on the predicted mean $\xPred{\ts}$ is not informed by the measurements $\meas_{\ts:\tsFin}$.
Previous work has noted that the approximation is sensitive to the linearisation point when the model is non-linear and the measurement noise is low \cite{prlf_bad_low_error}.

In the posterior linearisation approach \cite{pls}, we choose the optimal linearisation given the sequence of measurements and the resulting mean square error.
For the motion model in \cref{eq:def:state_space_model}, at time step $\ts$, we have
\begin{subequations}
\begin{align}
  \label{eq:linear_obj}
  &(\procA_{\ts}, \procB_{\ts}) = \underset{\procA_{\ts}^+, \procB_{\ts}^+}{\min\,} \expect{\twoNorm{\procFn_{\ts}(\x_{\ts}) - \procA_{\ts}^+ \x_{\ts} - \procB_{\ts}^+}^2  \mid \measSeq}, \\
  \label{eq:linear_obj:error_cov}
  \procLinCov_{\ts} =  &\expect{
    \left( \procFn_{\ts}(\x_{\ts}) - \procA_{\ts} \x_{\ts} - \procB_{\ts} \right)
    \left( \procFn_{\ts}(\x_{\ts}) - \procA_{\ts} \x_{\ts} - \procB_{\ts} \right)^{\top} \mid \measSeq
  }.
\end{align}
\end{subequations}
However, we cannot directly select parameters with respect to the unknown true posterior distribution.

Iterative smoothers take this insight into account to improve estimation performance.
By iteratively refining the linearisation in \cref{eq:def:lin:state_space_model}, the algorithms improve the estimates by using successively better linearisations.
Given estimates of the posterior moments $\xHatItSeq{i}, \PHatItSeq{i}$, we obtain a new linear approximation $\affAppParam{i+1}$,
from which new estimates of the moments $\xHatItSeq{i+1}, \PHatItSeq{i+1}$ are obtained using \rts smoothing \cite{bayes_filter_book}.
In summary, the following two-step process is iterated:
\begin{align}
  \label{eq:iterative_smoothers:iter}
  \affAppParam{i+1} &= \text{Linearisation}\left(\xHatItSeq{i}, \PHatItSeq{i}\right), \nonumber \\
  \left(\xHatItSeq{i+1}, \PHatItSeq{i+1}\right) &= \text{\rts smoother}\left(\affAppParam{i+1}, \measSeq \right).
\end{align}
The initial estimates of the moments are commonly, but not necessarily, the output of the corresponding non-iterative smoother.
Hopefully, with every iteration, the estimates of the posterior grow closer to the true posterior until the linearisation point is chosen with respect to the true posterior distribution.

The \ieks is a well-known iterative extension of the \eks \cite{bell_ieks}.
In the \ieks, the linearisation step in \cref{eq:iterative_smoothers:iter} is done with a first-order Taylor expansion around the estimated means $\xHatItSeq{i}$ from the previous iteration.
The \ipls is a more recent iterative smoother, first introduced in \cite{pls}.
In \ipls, $\affAppParam{i+1}$ is selected by performing \slr as in \cref{eq:slr:jacobian,eq:slr:offset,eq:slr:cov} and \cref{eq:slr:moments} with $p(\x_{\ts}) = \normal(\x_{\ts}; \xHatIt{i}_{\ts}, \PHatIt{i}_{\ts})$.

\subsection{The Gauss--Newton (\gn) method}
It is useful to view the iterated smoothers as optimisation algorithms, by identifying a cost function that they minimise.
An important advantage with this is that it enables us to make use of more general optimisation methods, with better convergence properties.

The Gauss--Newton (\gn) method \cite{num_opt}, is a well-known optimisation method which is used to solve problems on the form
\begin{align}
  \label{eq:gn:problem_form}
  \x^{*} &= \argMin{x}\, \gnObj(\x) = \argMin{x}\, \frac{1}{2} \twoNorm{\gnR(\x)}^2
  = \argMin{x}\, \frac{1}{2} \gnR(\x)^{\top} \gnR(\x),
\end{align}
where $\gnR(\x)$ is a given function.
Starting from an initial guess, the method iteratively finds the exact solution to the approximate objective
\begin{equation}
  \label{eq:gn:cost_fn_approx}
  \gnObjLin(\x) = \frac{1}{2} \twoNorm{\gnRLin(\x)}^2 = \frac{1}{2} \gnRLin(\x)^{\top} \gnRLin(\x),
\end{equation}
defined by the first order approximation of $\gnR(\cdot)$ around $\xHatIt{i}$
\begin{equation}
  \label{eq:gn:lin}
  \gnR(\x) \approx \gnRLin(\x) \coloneqq \gnR(\xHatIt{i}) + J_{\gnR}(\xHatIt{i}) (\x - \xHatIt{i}).
\end{equation}
Linearisation is done around the current iterate $\xHatIt{i}$ and the next iterate is the solution to the approximate problem \cref{eq:gn:cost_fn_approx} \cite{num_opt}.

\subsection{Levenberg--Marquardt regularisation}
The \gn method is not guaranteed to converge.
Instead, like many iterative methods, it can diverge or converge to a poor solution.
An extension of the \gn method with better convergence properties is the Levenberg--Marquardt method \cite{levenberg,marquardt,num_opt}.
In the \lm method, the standard cost function $\gnObj$ is extended with a regularisation term
\begin{align}
    \label{eq:lm:cost_fn}
    \lmObj{i}(\x) = \gnObj(\x)
    + \frac{1}{2} \lmRegParamIt{i} (\x - \xHatIt{i})^{\top} \left[ \regMat^{(i)}\right]^{-1} (\x - \xHatIt{i}),
\end{align}
where $\lmRegParamIt{i} > 0$ is an adaptable regularisation parameter and $\regMat^{(i)}$ is a sequence of positive definite regularisation matrices.
The matrices $\regMat^{(i)}$ can be selected to scale the problem suitably \cite{nonlinear_regression,num_opt}.
In this paper, we assume that these matrices are given while $\lmRegParamIt{i}$ is adapted as part of the optimisation algorithm.
The regularisation term encourages a new iterate to be close to the previous one, hopefully in a region where the approximation is acceptable.
A new iterate is accepted only if it decreases the cost function.
The level of regularisation is controlled by adapting $\lmRegParamIt{i}$ by
reducing $\lmRegParamIt{i}$ when an iterate is accepted and increasing it on rejection.
We introduce $\nu > 1$ as a parameter to control the adaptation and increase or reduce $\lmRegParamIt{i}$ with a factor $\nu$ on accepting or rejecting an iterate respectively \cite{marquardt,lm_adapt}.

\subsection{The \ieks as a \gn method}
\label{sec:background:ieks}
\noindent The \ieks can be interpreted as an optimisation method \cite{bell_ieks}.
Consider the problem of minimising the cost function
\begin{align}
  \label{eq:gn_ieks:cost_fn}
  &\gnIeksObj(\xSeq)
  = \frac{1}{2}\Big( (\x_1 - \priorMean)^{\top} \priorCov^{-1} (\x_1 - \priorMean) \nonumber \\
  &\quad+ \sum_{\ts=1}^\tsFin (\meas_{\ts} - h_{\ts}(\x_{\ts}))^{\top} \R_{\ts}^{-1} (\meas_{\ts} - h_{\ts}(\x_{\ts})) \nonumber \\
  &\quad+ \sum_{\ts=1}^{\tsFin-1} (\x_{\ts+1} - f_{\ts}(\x_{\ts}))^{\top} \Q_{\ts}^{-1} (\x_{\ts+1} - f_{\ts}(\x_{\ts})) \Big).
\end{align}
\gn optimisation of this function is equivalent to running the \ieks for the corresponding state-space model.
This was first shown in \cite{bell_ieks}, and can be stated as follows:

\begin{prop}
  \label[prop]{thm:gn_ieks}
  The \ieks inference of the state-space model in \cref{eq:def:state_space_model} is a \gn method for minimising the function $\gnIeksObj$ in \cref{eq:gn_ieks:cost_fn}.
\end{prop}

The idea of the proof is to linearise the cost function in \cref{eq:gn_ieks:cost_fn}
and compare it to the corresponding linearised IEKS state-space model.
Observing that GN and the IEKS exactly solve the same approximate problem,
we conclude that they are equivalent.
This proof structure is outlined in \cref{fig:props_overview} and we will reuse it when proving a similar connection between \gn and the \ipls.
\begin{figure}
  \centering
  \tikzstyle{var} = [
  rectangle,
  draw,
  text width=5em,
  text centered
]
\tikzstyle{line} = [midway, font=\small]
\begin{tikzpicture}[scale=0.8, node distance = 3cm, auto]
  \node [var] (L) {$\costFn = \frac{1}{2} \twoNorm{\gnR}^2$};
  \node [var, right=3.25cm of L] (LTilde) {$\tilde{\costFn}$};
  \node [var, below=1cm of L] (ssm) {State-space model};
  \node [var, below=0.85cm of LTilde] (assm) {Linear state-space model};
  \draw [->] (L) -- (LTilde) node [line] {Linearising $\gnR$};
  \draw [->] (ssm) -- (assm) node [line] {Linearising models};
  \draw [dashed, <->] (assm) -- (LTilde) node [line] {Neg. log. likelihood};
\end{tikzpicture}
  \caption{
    \label{fig:props_overview}
    Connection between \gn optimisation and a general Gaussian smoother.
    Both the \gn algorithm and the general Gaussian smoothers work by 1) linearising the problem and 2) solving the linearised problem analytically.
    The proofs of \cref{thm:gn_ieks,thm:gn_ipls} show that the linearised problems are identical.
  }
\end{figure}

\section{\lmIpls}
\subsection{\gn cost function}
\label{sec:our_method:ipls}
\noindent Inspired by the \ieks, we seek to establish a similar connection between the \ipls and the \gn method.
To this end, we require a cost function that leads to a \gn method, corresponding to the \ipls.

A key difference between the \ieks and the \ipls is the role the covariances play.
In the \ieks, the analytical linearisation is done only with respect to the estimated means $\xHatItSeq{i}$, whereas the \slr linearisation of the \ipls is based on both the estimated means $\xHatItSeq{i}$ and covariances $\PHatItSeq{i}$.
The \ieks further assumes that the linearisation error covariances are zero, whereas the \ipls estimates them as $\procLinCovIt{i}_{\ts}$ and $\measLinCovIt{i}_{\ts}$ respectively.
These matrices appear in the approximate state-space model of \cref{eq:def:lin:state_space_model} and must therefore be included in a cost function.
The estimated covariances $\PHatItSeq{i}$ are implicitly included in the approximate state-space model, since they are used in the \slr linearisation.

The inclusion of the covariance matrices complicates the construction of a matching \gn objective, which should be on a quadratic form in the state sequence $\xSeq$ (or some function of it).
Note that the state sequence $\xSeq$ is the optimisation variable and that the solution to the optimisation problem becomes the state estimates $\xHatItSeq{i+1}$.
The matrix defining the quadratic form will depend on the (estimated means of the) state sequence $\xHatItSeq{i}$ and the cost function will also, implicitly, depend on the covariance sequence $\PHatItSeq{i}$.

To establish the link between the \gn method and \ipls, we construct a cost function for the sequence of states fixing the covariances $\procLinCovIt{i}_{1:\tsFin}$, $\measLinCovIt{i}_{1:\tsFin}$ and $\PHatItSeq{i}$:
\begin{align}
  \label{eq:ipls:gn:cost_fn}
  &\gnIplsObj{i}(\xSeq)
  = \frac{1}{2} \Big( (\x_1 - \priorMean)^{\top} \priorCov^{-1} (\x_1 - \priorMean) \nonumber \\
  &+ \sum_{\ts=1}^{\tsFin-1} (\x_{\ts+1} - \stateBar_{\ts}(\x_{\ts}))^{\top} \left(\procCov_{\ts} + \procLinCovIt{i}_{\ts} \right)^{-1} (\x_{\ts+1} - \stateBar_{\ts}(\x_{\ts})) \nonumber \\
  &+ \sum_{\ts=1}^\tsFin (\meas_{\ts} - \measBar_{\ts}(\x_{\ts}))^{\top} \left( \measCov_{\ts} + \measLinCovIt{i}_{\ts} \right)^{-1} (\meas_{\ts} - \measBar_{\ts}(\x_{\ts})) \Big),
\end{align}
where $\stateBar(\cdot), \measBar(\cdot)$ the \slr estimated expectations of $\procFn_{\ts}$ and $\measFn_{\ts}$ in \cref{eq:slr:moments}
and $\procLinCovIt{i}_{\ts}, \measLinCovIt{i}_{\ts}$ are computed using \cref{eq:slr:cov} with respect to $\xHatItSeq{i}, \PHatItSeq{i}$ using $\procFn_{\ts}$ and $\measFn_{\ts}$ in \cref{eq:slr:moments} respectively.

An important property is that the cost function in \cref{eq:ipls:gn:cost_fn} depends on the most recent estimate of the sequence of covariance matrices,
both through the \slr expectations in \cref{eq:slr:moments} and through the estimated linearisation errors $\procLinCovIt{i}_{\ts}$ and $\measLinCovIt{i}_{\ts}$.

Before we derive the connection between \gn and \ipls we need to derive an expression for the gradient of our proposed cost function.
\begin{lemma}
  \label[lemma]{thm:gn_ipls:cost_fn:grad}
  The gradient of the \ipls cost function in \cref{eq:ipls:gn:cost_fn} is
  \begin{align}
    \label{eq:ipls:gn:cost_fn:grad}
    &\grad \gnIplsObj{i}(\xSeq) = J_{\gnRIt{i}}(\xSeq)^{\top} \gnRIt{i}(\xSeq),
  \end{align}
  where
  \begin{align}
  J_{\gnRIt{i}}(\xSeq)^{\top}
  &= \begin{pmatrix}
    (\procAIt{i}(\xSeq))^{\top} & (\measHIt{i})(\xSeq))^{\top}
  \end{pmatrix} \augQIplsInvSqrt,
  \end{align}
where $\procAIt{i}(\xSeq)$ and $\measHIt{i}(\xSeq)$ are given by
\cref{eq:def:ipls:proc_jacobian_seq} and \cref{eq:def:ipls:meas_jacobian_seq} in the appendix,
which also contains the proof of \cref{thm:gn_ipls:cost_fn:grad}.
\end{lemma}

We can then go on to present the main result of this section.
\begin{prop}
  \label[prop]{thm:gn_ipls}
The output of one iteration of \gn optimisation of the cost function $\gnIplsObj{i}(\xSeq)$ in \cref{eq:ipls:gn:cost_fn}, defined by the current \gn estimate $\xHatItSeq{i}$ and the covariance matrices $\PHatItSeq{i}$, is the same as the means estimated by \rts smoothing of \cref{eq:def:state_space_model}, linearised with \slr around $\xHatItSeq{i}, \PHatItSeq{i}$, that is, one iteration of the \ipls.
\end{prop}

\begin{proof}
The proof follows the steps outlined in \cref{fig:props_overview}.
We construct $\gnRIt{i}(\xSeq)$ such that $\gnIplsObj{i}(\xSeq) = \frac{1}{2} \twoNorm{\gnRIt{i}(\xSeq)}^2$ and then linearise $\gnRIt{i}(\xSeq)$ to form the approximate objective $\gnIplsObjLin{i}(\xSeq)$.
Secondly, the state-space model is linearised with \slr, according to the \ipls.
Finally, we compare the approximate \gn objective to the linearised state-space model and note that they correspond to the same minimisation problem.

Note that the covariance matrices $\procLinCovIt{i}_{\ts}$ and $\measLinCovIt{i}_{\ts}$ in $\gnIplsObj{i}(\xSeq)$ depend on the current estimates $\xHatItSeq{i}$ and not on the optimisation variable $\xSeq$.
Therefore, $\gnIplsObj{i}(\xSeq)$ is on the form of \cref{eq:gn:problem_form} and can be optimised with the \gn method.

We construct $\gnRIt{i}(\xSeq)$ by collecting states and measurements in a vector and grouping the covariance matrices in a single block diagonal matrix:
\begin{subequations}
\begin{align}
  \label{eq:def:ipls:aug_obs}
  \augObsIpls(\xSeq) &= \begin{pmatrix}
    \x_1 - \priorMean \\
    \x_2 - \stateBar_1(\x_1) \\
    \vdots \\
    \x_{\tsFin} - \stateBar_{\tsFin-1}(\x_{\tsFin-1}) \\
    \meas_1 - \measBar_1(\x_1) \\
    \vdots \\
    \meas_{\tsFin} - \measBar_{\tsFin}(\x_{\tsFin})
  \end{pmatrix}, \\
  \label{eq:def:ipls:augQ}
  \augQIplsInv &= \diag \Big(
  \priorCov^{-1}, (\procCov_1 + \procLinCovIt{i}_1)^{-1}, \dots, (\procCov_{\tsFin-1} + \procLinCovIt{i}_{\tsFin-1})^{-1}, \nonumber \\
                 &\hspace{1.5cm}(\measCov_1 + \measLinCovIt{i}_1)^{-1}, \dots, (\measCov_{\tsFin} + \measLinCovIt{i}_{\tsFin})^{-1} \Big).
\end{align}
\end{subequations}
Defining
\begin{equation}
  \label{eq:def:ipls:gnr}
  \gnRIt{i}(\xSeq) = \augQIplsInvSqrtT \augObsIpls(\xSeq),
\end{equation}
with $\augQIplsInvSqrt \augQIplsInvSqrtT = \augQIplsInv$.

Next, we linearise $\gnRIt{i}$ as the first-order approximation $\gnRLin$ around $\xHatItSeq{i}$
\begin{align}
  \label{eq:ipls:gnr:taylor}
  \gnRLin(\xSeq)
             &= \gnRIt{i}(\xHatItSeq{i}) + J_{\gnRIt{i}}(\xHatItSeq{i}) (\xSeq - \xHatItSeq{i})
             = \augQIplsInvSqrtT \augObsIplsApprox(\xSeq),
\end{align}
where
\begin{align*}
  \augObsIplsApprox(\xSeq) =
    \begin{pmatrix}
    \x_1 - \priorMean \\
    \x_2 - \left( \procAIt{i}_1(\xHatIt{i}_1) \x_1 + \procB_{1}(\xHatIt{i}_1) \right ) \\
    \vdots \\
    \x_{\tsFin} - \left( \procAIt{i}_{\tsFin-1}(\xHatIt{i}_{\tsFin-1}) \x_{\tsFin-1} + \procB_{\tsFin-1}(\xHatIt{i}_{\tsFin-1}) \right ) \\
    \meas_1 - \left( \measHIt{i}_{1}(\xHatIt{i}_1) \x_1 + \measC_{1}(\xHatIt{i}_1) \right) \\
    \vdots \\
    \meas_{\tsFin} - \left( \measHIt{i}_{\tsFin}(\xHatIt{i}_{\tsFin}) \x_{\tsFin} + \measC_{\tsFin}(\xHatIt{i}_{\tsFin}) \right) \\
    \end{pmatrix}
\end{align*}
and $\procAIt{i}_{\ts}(\x_{\ts}), \measHIt{i}_{\ts}(\x_{\ts})$ are the \slr Jacobians of $\procFn_{\ts}, \measFn_{\ts}$ respectively,
see the appendix for the derivation details.

The approximate \gn objective becomes
\begin{align}
  \label{eq:ipls:approx_objective}
  \gnIplsObjLin{i}(\xSeq)
  &= \frac{1}{2}\left(\augObsIplsApprox(\xSeq)\right)^{\top} \augQIplsInv \augObsIplsApprox(\xSeq) \left(\augObsIplsApprox(\xSeq)\right).
\end{align}
To perform one iteration of the \ipls, we instead first linearise the state-space model in \cref{eq:def:state_space_model} using \slr with respect to $\normal(\x_{\ts}; \xHatIt{i}_{\ts}, \PHatIt{i}_{\ts})$.
The resulting approximate state-space model is on the form \cref{eq:def:lin:state_space_model} with linearisation parameters $\affAppParam{i}$ selected using \cref{eq:slr:jacobian,eq:slr:offset,eq:slr:cov}.
The next iterate $\xHatItSeq{i+1}$ is computed as the closed-form output of \rts smoothing.

Examining $\gnIplsObjLin{i}(\xSeq)$, we note that it is the negative log-posterior of the \slr linearised state-space model (up to a constant).
The next \gn iterate will be the closed-form solution to this minimisation problem.
Since the two methods, \gn and \ipls, compute in a single iteration the exact solution to the same optimisation problem, their output $\xHatItSeq{i+1}$ must be the same.
\end{proof}
\subsection{Levenberg--Marquardt regularisation}
\label{sec:our_method:lm}
\noindent The now established connection between the \ieks and \ipls and \gn optimisation makes the \lm method a promising alternative for regularisation.
\Cref{thm:lm_rts_smoothing} shows how \lm-regularisation can be achieved through smoothing of a slightly modified state-space model.
The result was shown for the \lmIeks in \cite{lm_ieks} and is here generalised to include the \lmIpls.
Similar interpretations of regularisation as extra measurements are discussed in \cite{lm_ieks_2,lm_observation},
but only for the \ieks.
\begin{prop}
  \label[prop]{thm:lm_rts_smoothing}
  Iterated smoothing with \ieks or \ipls (under the conditions in \cref{thm:gn_ipls}) for a state-space model as in \cref{eq:def:state_space_model}, extended with the measurement
  \begin{equation}
    \label{eq:lm:extended_problem}
    \xHatIt{i}_{\ts} = \x_{\ts} + \lmErr_{\ts},\qquad \lmErr_{\ts} \sim \normal(0, (\lmRegParamIt{i})^{-1} \regMat^{(i)}_{\ts})
  \end{equation}
  is a \gn method with the \lm-regularisation defined in \cref{eq:lm:cost_fn}, if $\regMat^{(i)}$ is a sequence of block-diagonal regularisation matrices:
  $\regMat^{(i)} = \diag \left(\regMat_1^{(i)}, \dots, \regMat_{\tsFin}^{(i)} \right)$, $\regMat^{(i)}_{\ts} \in \reals^{\dimState \times \dimState}, \forall \ts = 1, \dots, \tsFin$.
\end{prop}

\begin{proof}
The proof follows the structure of the earlier proofs and we can use the results of \cref{thm:gn_ieks,thm:gn_ipls} to simplify it.
First, we construct $\gnRLM{i}(\xSeq)$ such that $\lmObj{i}(\xSeq) = \frac{1}{2} \twoNorm{\gnRLM{i}(\xSeq)}^2$ and
show that the linearisation of $\gnRLM{i}(\xSeq)$ results in an approximate objective which is $\gnObjLin(\xSeq)$ plus an extra term.
Second, we introduce the measurement in \cref{eq:lm:extended_problem} into the state-space model \cref{eq:def:state_space_model} and linearise.
Third, we confirm that the \lm optimisation and iterative smoothers solve the same minimisation problem at each iteration.
The full proof is given in the appendix.
\end{proof}
In other words, we achieve an \lm-regularised version of the smoother by imposing this additional modelling assumption.




%

\subsection{\lmIpls Algorithm}
\label{sec:our_method:lm_alg}
\noindent Here, we present a full description of the \lm-regularised iterative smoothers.
The method is an iterative smoother which uses estimates from the previous iteration to linearise the motion and measurement models, giving an approximate affine state-space model as in \cref{eq:def:lin:state_space_model}.
A new estimate is then proposed through \rts smoothing of the affine state-space model, with an extra measurement of the state, corresponding to \lm regularisation, see \cref{thm:lm_rts_smoothing}.
The new estimate is accepted if it results in a lower value for an associated cost function, see \cref{eq:gn_ieks:cost_fn,eq:ipls:gn:cost_fn}.

A single iteration step for the linearised models is described in \cref{alg:lm_iter}.
\begin{algorithm}[t]
  \caption{\lm smoother single inner loop iteration}
  \label{alg:lm_iter}
  \small
\begin{algorithmic}[1]
  \Require
  The current estimated means $\xHatItSeq{i}$,
  priors $\priorMean$ and $\priorCov$,
  measurements $\measSeq$,
  affine approximations of motion and meas. models $\affAppParam{i}$,
  regularisation parameter $\lambda$, regularisation matrices $\regMat_{1:\tsFin}$,
  and implicitly a cost function $\gnObj$.

  \Ensure {
    New smoothed estimated means and covariances $\xHatSeqSmooth{s}, \PHatSeqSmooth{s}$,
    s.t. $\gnObj(\xHatSeqSmooth{s}) < \gnObj(\xHatItSeq{i})$, and updated $\lambda$.
  }
  \Procedure{\lm-it.}{$\xHatItSeq{i}, \PHatItSeq{i}, \measSeq, \lambda$}
    \Repeat \, // Until \lm cost reduction
      \For{$\ts=1, \ldots, \tsFin$}
        \State $\xPred{\ts},\ \PPred{\ts} \gets \textsc{\kf prediction}$
        \State // Standard update based on $\meas_{\ts}$.
        \State $\xUpd{\ts},\ \PUpd{\ts} \gets \textsc{\kf update}$

        \State // Extra \lm-regularisation update step:
        \If{$\lambda > 0$}
          \State  $\innCov_{\ts} \gets \PUpd{\ts} + \lambda^{-1} \regMat_{\ts}$
          \State $\K_{\ts} \gets \PUpd{\ts} \innCov_{\ts}^{-1}$
          \State $\xUpd{\ts} \gets \xUpd{\ts} + \K_{\ts} [\xHatIt{i}_{\ts} - \xUpd{\ts} ]$
          \State $\PUpd{\ts} \gets \PUpd{\ts} - \K_{\ts} \innCov_{\ts} [\K_{\ts}]^{\top}$
        \EndIf
      \EndFor
      \State $\xHatSeqSmooth{s},\ \PHatSeqSmooth{s} \gets \textsc{\rts smoothing}$
      \If{$\gnObj(\xHatSeqSmooth{s}) < \gnObj(\xHatItSeq{i})$}
          \State // Decrease regularisation and accept the iterate.
          \State $\lambda \gets \lambda / \nu$
          \Else
          \State // Increase regularisation and reject the iterate.
          \State $\lambda \gets \nu\lambda$
      \EndIf
    \Until{the iterate is accepted}
    \State \Return $\xHatSeqSmooth{s},\ \PHatSeqSmooth{s}, \lambda$
  \EndProcedure
\end{algorithmic}
\end{algorithm}

The full algorithm is simply the iteration of steps taken in \cref{alg:lm_iter}, using the accepted estimates in the previous step as the estimates used for linearisation.
The complete procedure is described in \cref{alg:lm_smoother}.
\begin{algorithm}[t]
    \caption{\lmIeks and \lmIpls algorithms}
    \label{alg:lm_smoother}
    \small
    \begin{algorithmic}[1]
      \Require
      Initial moments $\xHatItSeq{0}$, $\PHatItSeq{0}$,
      priors $\priorMean$ and $\priorCov$,
      measurements $\meas_{1:\tsFin}$,
      increase/decrease parameter $\nu> 1$,
      initial regularisation parameter $\lmRegParamInit$,
      smoother type $t \in \{\text{\lmIeks, \lmIpls}\}$
      and implicitly
      a cost function $\lmObj{i}$,
      motion and measurement models and parameters:
      $\procFn_{1:\tsFin}$, $\procCov_{1:\tsFin}$,
      $\measFn_{1:\tsFin}$, $\measCov_{1:\tsFin}$
      and regularisation matrices $\regMat_{1:\tsFin}$,
      \Ensure The smoothed trajectory $\xHatSeqSmooth{*},\ \PHatSeqSmooth{*}$.
        \Procedure{\lmIpls/\lmIeks}{}\
        \State Set $i \gets 0$ and $\lmRegParamIt{i} \gets \lmRegParamInit$
        \Repeat
          \If{$t == $ \lmIeks}
            \State Set $\procLinCov_{1:\tsFin}, \measLinCov_{1:\tsFin}$ to 0, see \cref{eq:ext:cov}
          \ElsIf{$t == $ \lmIpls}
            \State Calc. $\procLinCov_{1:\tsFin}, \measLinCov_{1:\tsFin}$ using $\xHatItSeq{i}, \PHatItSeq{i}$ in \cref{eq:slr:cov}
          \EndIf
          \Repeat
            \If{$t == $ \lmIeks}
              \State // Affine approx. using \cref{eq:ext:jacobian,eq:ext:offset}.
              \State Calc. $\procA_{1:\tsFin}, \procB_{1:\tsFin}, \measH_{1:\tsFin}, \measC_{1:\tsFin}$ using $\xHatItSeq{i}$
            \ElsIf{$t == $ \lmIpls}
              \State // Affine approx. using \cref{eq:slr:jacobian,eq:slr:offset}.
              \State Calc. $\procA_{1:\tsFin}, \procB_{1:\tsFin}, \measH_{1:\tsFin}, \measC_{1:\tsFin}$ using $\xHatItSeq{i}, \PHatItSeq{i}$
            \EndIf
            \State $\affAppParam{i} \gets
            \procA_{1:\tsFin}, \procB_{1:\tsFin}, \procLinCov_{1:\tsFin},
            \measH_{1:\tsFin}, \measC_{1:\tsFin}, \measLinCov_{1:\tsFin}$

            \State $\xHatSeqSmooth{s}, \PHatSeqSmooth{s}, \lmRegParamIt{i} \gets \textsc{\lm-it.}(
            \xHatItSeq{i}, \measSeq, \lmRegParamIt{i}, \affAppParam{i})$
            \State //The $\PHatItSeq{i}$ estimate is kept in the inner loop.
            \State $\xHatItSeq{i+1},\ \PHatItSeq{i+1}, \lmRegParamIt{i+1} \gets \xHatSeqSmooth{s}, \PHatItSeq{i}, \lmRegParamIt{i}$
            \State $i \gets i+1$
          \Until{inner loop termination condition met}
          \State // The covariance estimates are updated here.
          \State $\PHatItSeq{i} \gets \PHatSeqSmooth{s}$
        \Until{convergence}
        \State \Return $\xHatItSeq{i},\ \PHatItSeq{i}$
     \EndProcedure
    \end{algorithmic}
\end{algorithm}

For readability, we omit some model parameters in the algorithmic description, the origins of which should be clear from the context.

The differences between the variants of the \lm smoother stems from the different method of linearisation:
\slr defined in \cref{eq:slr:jacobian,eq:slr:offset,eq:slr:cov} for the \lmIpls and Taylor expansion in \cref{eq:ext:jacobian,eq:ext:offset,eq:ext:cov} for the \lmIeks.
Apart from the obvious difference in the computed linearisation, the \slr also requires some extra steps in the algorithm, which we detail below.

The \ipls's use of both the estimated means and covariances for the linearisation requires a sequence of cost functions (instead of a single cost function), see \cref{sec:our_method:ipls}.
The cost function is changed when the current estimated covariances are updated.
In practice, the algorithm controls this by adding an inner loop in \cref{alg:lm_smoother}.

The inner loop allows for an arbitrary number of \lm-iteration steps, where the estimated means are updated while the covariances are kept fixed.
After some provided termination condition is fulfilled, the covariances are updated, thereby moving on to a new cost function.
We have found that the simplest setting, to exit the inner loop after a single iteration, works well in practice but more elaborate conditions are possible,
such as requiring a sufficient decrease in $\lmObj{i}$ or a sufficiently large $\lmRegParamIt{i}$.
For the \lmIeks this inner loop has no effect since it uses the same cost function throughout the optimisation.

In \cref{alg:lm_smoother} the \lmIeks and \lmIpls differ only in the different methods of linearisation that are applied when estimating the affine approximations.

\section{\lsIpls}
\label{sec:our_method:line_search}
\noindent Another way to improve the \gn method is to introduce a line-search (LS) procedure to the algorithm \cite{num_opt}.
The LS version of the \ieks was described in \cite{lm_ieks} and here we extend it to the \ipls,
re-using the cost function in \cref{eq:ipls:gn:cost_fn}.
LS can be implemented by introducing a parameter $\alpha > 0$ to restrict the iterative update of the estimates.
That is, given $\xHatItSeq{i+1}$ and $\xHatItSeq{i}$,
we define $\searchDirSeq{i} = \xHatItSeq{i+1} - \xHatItSeq{i}$, and we obtain the LS update of the estimates:
\begin{subequations}
\begin{align}
  \label{eq:ls:update}
  \xHatItSeq{i+1}(\alpha) &= \xHatItSeq{i} + \alpha \searchDirSeq{i}, \\
  \label{eq:ls:step_length}
  \alpha &= \argMin{\alpha' \in [0,1]} \gnObj^{(i)}(\xHatItSeq{i} + \alpha' \searchDirSeq{i}).
\end{align}
\end{subequations}

\subsection{Armijo--Wolfe conditions for \lsIpls}
We can also use an inexact version of LS, where we seek a point
$\lsPoint = \xHatItSeq{i} + \alpha \searchDirSeq{i}$,
for which we only require a sufficient decrease in the cost function.
This is guaranteed by fulfilling the Armijo or Wolfe conditions \cite{num_opt}
\begin{subequations}
\begin{align}
  \label{eq:def:armijo_cond}
  \costFn(\lsPoint) \leq \costFn(\xHatItSeq{i}) + c_1 \alpha (\searchDirSeq{i})^{\top} \grad \costFn(\xHatItSeq{i}), \\
  \label{eq:def:wolfe_cond}
  (\searchDirSeq{i})^{\top} \grad \costFn(\lsPoint) \geq c_2 (\searchDirSeq{i})^{\top} \grad \costFn(\xHatItSeq{i}),
\end{align}
\end{subequations}
where $0 < c_1 < c_2 < 1$.

\begin{lemma}
  \label[lemma]{thm:ls:armijo_wolfe:ipls}
  An efficient computation of the directional derivatives for the \ipls cost function in \cref{eq:ipls:gn:cost_fn} to evaluate the Armijo--Wolfe conditions in \cref{eq:def:armijo_cond,eq:def:wolfe_cond}
  is given by
  \begin{align}
    &(\searchDirSeq{i})^{\top} \grad \costFn(\xHatItSeq{i})
    = (\searchDir{i}_{1})^{\top} \priorCov^{-1}(\xHatIt{i}_1 - \priorMean) \nonumber \\
    &\hspace{1cm}+ \sum_{\ts=1}^{K-1} \left( \searchDir{i}_{\ts+1} - \procA_{\ts}(\xHatIt{i}_\ts) \searchDir{i}_{\ts} \right)^{\top}
    \times (\procCov_{\ts} + \procLinCovIt{i}_{\ts})^{-1}(\xHatIt{i}_{\ts+1} - \stateBar_\ts(\xHatIt{i}_\ts)) \nonumber \\
    &\hspace{1cm}- \sum_{\ts=1}^{K} \left( \measH_{\ts}(\xHatIt{i}_\ts) \searchDir{i}_{\ts} \right)^{\top} (\measCov_{\ts} + \measLinCovIt{i}_{\ts})^{-1}(\xHatIt{i}_{\ts+1} - \measBar_\ts(\xHatIt{i}_\ts)). \nonumber \\
  \end{align}
\end{lemma}
\begin{proof}
To evaluate the Armijo--Wolfe conditions for a certain step length $\alpha$,
we need to compute the gradient of the cost function in \cref{eq:ipls:gn:cost_fn},
which comes directly from \cref{thm:gn_ipls:cost_fn:grad}.
The derivations of the directional derivatives are in the appendix.
\end{proof}

By selecting a suitable estimate on a line between the previous estimate and the new one proposed by the smoothing iteration,
the methods improve since the size of the update step is allowed to decrease for iteration updates that risk diverging.
Potentially, this could also lead to faster convergence, albeit with the extra computational demand incurred by finding a suitable $\alpha$.

\subsection{\lsIpls Algorithm}
\label{sec:our_method:ls_alg}
\noindent Here, we present a full description of the line-search iterative smoothers.
The basis of the line-search algorithms \lsIeks and \lsIpls is to optimise the line that connects the previous and proposed estimates, see \cref{eq:ls:update,eq:ls:step_length}.
However, when we use the \ieks/\ipls implementation of the \gn method, there is no increment computed in the same sense as in the classical formulation of the \gn method.
Fortunately, given the previous iterate $\xHatItSeq{i}$ and the proposed iterate $\xHatSeqSmooth{s}$, we can compute the corresponding increment via $\Delta \xHatItSeq{i+1} = \xHatItSeq{i+1} - \xHatSeqSmooth{s}$.
The proposed iterate $\xHatSeqSmooth{s}$ is computed with a standard step of the \gn optimisation, which is equivalent to running \cref{alg:lm_iter} with $\lmRegParamIt{i} = 0$.
An inexact line-search algorithm is described in \cref{alg:line_search}.

\begin{algorithm}[t]
    \caption{Inexact \lsIeks and \lsIpls algorithms}
    \label{alg:line_search}
    \small
    \begin{algorithmic}[1]
        \Require
        Initial moments $\xHatItSeq{0}$, $\PHatItSeq{0}$,
        priors $\priorMean$ and $\priorCov$,
        measurements $\meas_{1:\tsFin}$,
        smoother type $t \in \{\text{\lmIeks, \lmIpls}\}$
        and implicitly
        a cost function $\costFn^{(i)}$,
        motion and measurement models and parameters:
        $\procFn_{1:\tsFin}$, $\procCov_{1:\tsFin}$,
        $\measFn_{1:\tsFin}$, $\measCov_{1:\tsFin}$.
      \Ensure The smoothed trajectory $\xHatSeqSmooth{*},\ \PHatSeqSmooth{*}$.
      \Procedure{\lsIeks/\lsIpls}{}\
        \State Set $i \gets 0$
        \Repeat
          \If{$t == $ \lsIeks}
            \State Set $\procLinCov_{1:\tsFin}, \measLinCov_{1:\tsFin}$ to 0, see \cref{eq:ext:cov}
          \ElsIf{$t == $ \lsIpls}
            \State Est. $\procLinCov_{1:\tsFin}, \measLinCov_{1:\tsFin}$ using $\xHatItSeq{i}, \PHatItSeq{i}$ in \cref{eq:slr:cov}
          \EndIf
          \State // Initial covariances in the inner loop.
          \State $\PHatSeqSmooth{s'} \gets \PHatItSeq{i}$
          \Repeat
            \If{$t == $ \lsIeks}
              \State // Affine approx. using \cref{eq:ext:jacobian,eq:ext:offset}.
              \State Est. $\procA_{1:\tsFin}, \procB_{1:\tsFin}, \measH_{1:\tsFin}, \measC_{1:\tsFin}$ using $\xHatItSeq{i}$
            \ElsIf{$t == $ \lsIpls}
              \State // Affine approx. using \cref{eq:slr:jacobian,eq:slr:offset}.
              \State Est. $\procA_{1:\tsFin}, \procB_{1:\tsFin}, \measH_{1:\tsFin}, \measC_{1:\tsFin}$ using $\xHatItSeq{i}, \PHatItSeq{i}$
            \EndIf
            \State $\affAppParam{i} \gets
            \procA_{1:\tsFin}, \procB_{1:\tsFin}, \procLinCov_{1:\tsFin},
            \measH_{1:\tsFin}, \measC_{1:\tsFin}, \measLinCov_{1:\tsFin}$
            \State // \lm-iter with $\lmRegParamIt{i} = 0$ corresp. to a \gn-step.
            \State $\xHatSeqSmooth{s},\ \PHatSeqSmooth{s} \gets \textsc{\lm-iter}(\xHatItSeq{i}, \PHatItSeq{i}, \measSeq, 0, \affAppParam{i})$
            \State $\searchDir{i}, \Delta \PHatSeqSmooth{s} \gets \xHatSeqSmooth{s} - \xHatItSeq{i}, \PHatSeqSmooth{s} - \PHatSeqSmooth{s'}$
            \State Select $\alpha$ satisfying the Armijo--Wolfe cond.
            \State $\xHatItSeq{i+1} \gets \xHatItSeq{i} + \alpha \searchDir{i}$
            \State // Update covariances
            \State $\PHatSeqSmooth{s'} \gets \PHatSeqSmooth{s'} + \alpha \Delta \PHatSeqSmooth{s}$
            \State $\PHatItSeq{i+1} \gets \PHatItSeq{i}$
            \State $i \gets i+1$
          \Until{inner loop termination condition met}
          \State $\PHatItSeq{i} \gets \PHatSeqSmooth{s'}$
        \Until{Converged}
        \State \Return $\xHatItSeq{i},\ \PHatItSeq{i}$
      \EndProcedure
    \end{algorithmic}
\end{algorithm}

Similar to the \lm-regularised methods, there are some differences between the \ieks and \ipls -based versions of the line-search algorithm.
The \lsIeks and \lsIpls use their respective version of the smoother iteration in \cref{alg:lm_iter}, detailed in \cref{sec:our_method:lm_alg}.
For the \lsIeks, only the estimated means $\xHatItSeq{i}$ are used in the linearisation and the estimated covariances are disregarded.
For the same reason, the inner loop which enables repeated optimisation of the same cost function, that is with covariances kept fixed, has no effect for the \lsIeks and it exits the loop after a single iteration.
%

\section{Simulation results}
\label{sec:exp}
\noindent We demonstrate the benefits of the \lm regularised and line-search smoothers in highly nonlinear smoothing problems.
In the simulations, we do a single iteration of the inner loop, meaning that the cost function is updated at every iteration.
For the \lmIpls, we use $\lmRegParamIt{0} = 0.01, \nu=10$, $\regMat^{(i)} = \diag(\regMat_1^{(i)}, \dots, \regMat_{\tsFin}^{(i)})$ with $\regMat^{(i)}_{\ts} = \mat{I},\ \ts = 1, \dots, \tsFin$.
For the \lsIpls, we perform inexact line-search with Armijo--Wolfe conditions, with $c_1 = 0.1, c_2 = 0.9$.
The experiments are implemented in Python and the code is publicly available\footnote{\href{https://www.github.com/jackonelli/post_lin_smooth}{\texttt{github.com/jackonelli/post\_lin\_smooth}}}.
\subsection{Coordinated turn (CT) model with bearings only measurements}
\label{sec:exp:ct_const_sens}
\noindent We extend the experiment from \cite{lm_ieks} to include the \ipls methods along with the \ieks method of the original paper.
The experiment setup is a sequence of true states $\xSeq$ of length $\tsFin=500$, simulated from a coordinated turn model and measurements of bearings only.
See the appendix for a full specification of the CT model.
The bearings measurements come from two sensors placed at $(-1.5, 0.5)^{\top}$ and $(1, 1)^{\top}$ respectively, with relatively high noise with variance $\sigma^2 = 1/2^2 \ \rad^2$.

A single realisation is shown in \cref{fig:ct:realisation}, along with examples of estimated trajectories from the different models.
\begin{figure}
  \centering
\begin{tikzpicture}[scale=0.8]

\begin{axis}[
legend cell align={left},
legend columns=1,
legend style={fill opacity=0.8, , at={(1.01,1)}, anchor=north west, draw opacity=1, text opacity=1, draw=white!80!black},
tick align=outside,
tick pos=left,
x grid style={white!69.0196078431373!black},
xlabel={$x(1)$},
xmin=-1.6, xmax=1.1,
xtick style={color=black},
y grid style={white!69.0196078431373!black},
ylabel={$x(2)$},
ymin=-1, ymax=1.1,
ytick style={color=black},
cycle list name=colorAndLine
]
\addplot+[semithick]
table [x, y]{fig/ieks_realisation/ieks.data};
\addlegendentry{\ieks};
\addplot+[semithick]
table [x, y]{fig/ieks_realisation/lm-ieks.data};
\addlegendentry{\lmIeks};
\addplot+[semithick]
table [x, y]{fig/ieks_realisation/ls-ieks.data};
\addlegendentry{\lsIeks};
\addplot+[semithick]
table [x, y]{fig/ieks_realisation/ipls.data};
\addlegendentry{\ipls}
\addplot+[semithick]
table [x, y]{fig/ieks_realisation/lm-ipls.data};
\addlegendentry{\lmIpls}
\addplot+[semithick]
table [x, y]{fig/ieks_realisation/ls-ipls.data};
\addlegendentry{\lsIpls}
\addplot [semithick, black, mark=*, mark size=0.25, mark options={solid}, only marks]
table [x, y]{fig/ieks_realisation/true_traj.data};
\addlegendentry{$\xSeq$};
\addplot [semithick, black, mark=triangle*, mark size=2, mark options={solid}, only marks, forget plot]
table [x, y]{
x y
-1.5 0.5
1 1
};
\end{axis}
\end{tikzpicture}
  \caption{\label{fig:ct:realisation}CT experiment with bearings only measurements. The two sensors are placed at $(-1.5, 0.5)^{\top}$ and $(1, 1)^{\top}$.}
\end{figure}
For this particular realisation, all algorithms perform similarly, largely following the true trajectory.

To see a discrepancy between the models we repeat the experiment for \numMcRuns independent trials,
where the true trajectory and measurements are resampled at every trial.
For each method, we measure the \textit{root mean square error (\rmse)} and \textit{normalised estimation error squared (\nees)} \cite{app:nav_and_track}, across 10 iterations.
The results are displayed in \cref{fig:ct:bearings_only:metrics}.
\begin{figure}
  \centering
  \begin{tikzpicture}
\pgfplotsset{/pgfplots/error bars/error bar style={very thick}}

\begin{groupplot}[
    group style={
        group size=2 by 1,
        horizontal sep=1.5cm,
    },
    xlabel={Iteration},
    height=5cm,
    width= 6.5cm,
    xmin=0.5, xmax=10.5,
    ylabel near ticks,
]
\nextgroupplot[
legend cell align={left},
legend columns=6,
legend style={fill opacity=0.8, draw opacity=1, text opacity=1, at={(1,1.05)}, anchor=south, draw=white!80!black},
tick align=outside,
tick pos=left,
x grid style={white!69.0196078431373!black},
xtick style={color=black},
ymode=log,
ymin=0.3, ymax=10,
ylabel={\rmse},
ytick style={color=black},
y grid style={white!69.0196078431373!black},
cycle list name=onlyColor,
ylabel shift=-5pt,
]
\addplot+[error bars/.cd, y dir=both, y explicit]
table[x=x, y=y, y error=err] {fig/ct_bearings_only_metrics/rmse/ieks.data};
\addlegendentry{\ieks}
\addplot+[error bars/.cd, y dir=both, y explicit]
table[x=x, y=y, y error=err] {fig/ct_bearings_only_metrics/rmse/lm-ieks.data};
\addlegendentry{\lmIeks}
\addplot+[error bars/.cd, y dir=both, y explicit]
table[x=x, y=y, y error=err] {fig/ct_bearings_only_metrics/rmse/ls-ieks.data};
\addlegendentry{\lsIeks}
\addplot+[error bars/.cd, y dir=both, y explicit]
table[x=x, y=y, y error=err] {fig/ct_bearings_only_metrics/rmse/ipls.data};
\addlegendentry{\ipls}
\addplot+[error bars/.cd, y dir=both, y explicit]
table[x=x, y=y, y error=err] {fig/ct_bearings_only_metrics/rmse/lm-ipls.data};
\addlegendentry{\lmIpls}
\addplot+[error bars/.cd, y dir=both, y explicit]
table[x=x, y=y, y error=err] {fig/ct_bearings_only_metrics/rmse/ls-ipls.data};
\addlegendentry{\lsIpls}
\nextgroupplot[
ymode=log,
tick align=outside,
tick pos=left,
ylabel={\nees},
x grid style={white!69.0196078431373!black},
xtick style={color=black},
y grid style={white!69.0196078431373!black},
ymin=2, ymax=1200,
ytick style={color=black},
cycle list name=onlyColor,
ylabel shift=-5pt,
]
\addplot+[error bars/.cd, y dir=both, y explicit]
table[x=x, y=y, y error=err] {fig/ct_bearings_only_metrics/nees/ieks.data};
\addplot+[error bars/.cd, y dir=both, y explicit]
table[x=x, y=y, y error=err] {fig/ct_bearings_only_metrics/nees/lm-ieks.data};
\addplot+[error bars/.cd, y dir=both, y explicit]
table[x=x, y=y, y error=err] {fig/ct_bearings_only_metrics/nees/ls-ieks.data};
\addplot+[error bars/.cd, y dir=both, y explicit]
table[x=x, y=y, y error=err] {fig/ct_bearings_only_metrics/nees/ipls.data};
\addplot+[error bars/.cd, y dir=both, y explicit]
table[x=x, y=y, y error=err] {fig/ct_bearings_only_metrics/nees/lm-ipls.data};
\addplot+[error bars/.cd, y dir=both, y explicit]
table[x=x, y=y, y error=err] {fig/ct_bearings_only_metrics/nees/ls-ipls.data};
\end{groupplot}
\end{tikzpicture}
  \caption{
    \label{fig:ct:bearings_only:metrics}Simulated CT model with bearings only measurements, see \cref{fig:ct:realisation} for setup.
  Curves show averaged \rmse and \nees across iterations averaged over \numMcRuns trials.
  Error bars correspond to the standard error, i.e. the estimated standard deviation scaled by $1 / \sqrt{\numMcRuns}$.
  }
\end{figure}

From these results, it is clear that the \ipls methods consistently perform better and require fewer iterations to reach a good trajectory.
The large spread in the metrics for the \ieks methods comes from the fact that they diverge for a significant number of realisations.
Among the \ipls methods there is little variation, with \lmIpls and \lsIpls possibly showing a slightly faster reduction in \rmse;
this problem has little need for regularisation beyond what the standard \ipls provides.

The results indicate that the \ipls cost function is a better optimisation objective, compared to the MAP objective of the \ieks, at least in terms of \rmse and \nees.
This advantage has the obvious caveat that a single iteration of the \ipls methods is more computationally expensive than its \ieks counterpart.
\subsection{CT model with time-dependent bearings only measurement model}
\label{sec:exp:ct_var_sens}
\noindent To examine the impact of regularisation, we analyse a special case of the coordinated turn experiment in \cref{sec:exp:ct_const_sens}.

The experiment setup is almost identical to the original experiment above, with a true trajectory of length $\tsFin=500$, simulated from a coordinated turn model and a bearings-only measurement model.
The bearings measurements come from two sensors placed at $(-1.5, 0.5)^{\top}$ and $(1, 1)^{\top}$ respectively, both with relatively high noise with variance $\sigma^2 = 1/2^2 \ \rad^2$.
To highlight the benefit of regularisation we modify the sensor arrangement to create some challenging non-linearities:
For time steps $k = 50, 100, \dots, 500$, the measurement consists of a single reading from the sensor at $(1, 1)^{\top}$, but with a low noise with $\sigma^2 = 0.025^2 \ \rad^2$.

For a more stable simulation, we use fixed initial estimates for all iterated smoothers
%
\begin{equation*}
  \xHat^{(0)}_{\ts} = \vek{0},\, \PHat^{(0)}_{\ts} = \priorCov,\, \ts = 1, \dots, K.
\end{equation*}
A single realisation is shown in \cref{fig:ct_bearings_only_varying_sensors:setup}, along with examples of estimated trajectories from the different models.
\begin{figure}
  \centering
  \begin{tikzpicture}[scale=0.8]
\begin{axis}[
legend cell align={left},
legend columns=1,
legend style={fill opacity=0.8, , at={(1.01,1)}, anchor=north west, draw opacity=1, text opacity=1, draw=white!80!black},
tick align=outside,
tick pos=left,
title={Estimates},
x grid style={white!69.0196078431373!black},
xlabel={$x(1)$},
xmin=-0.2, xmax=0.9,
xtick style={color=black},
y grid style={white!69.0196078431373!black},
ylabel={$x(2)$},
ymin=-1.2, ymax=0.75,
ytick style={color=black},
cycle list name=colorAndLine
]
\addplot+[semithick]
table [x, y]{fig/ct_bearings_only_varying_sensors/setup/ieks.data};
\addlegendentry{\ieks};
\addplot+[semithick]
table [x, y]{fig/ct_bearings_only_varying_sensors/setup/lm-ieks.data};
\addlegendentry{\lmIeks};
\addplot+[semithick]
table [x, y]{fig/ct_bearings_only_varying_sensors/setup/ls-ieks.data};
\addlegendentry{\lsIeks};
\addplot+[semithick]
table [x, y]{fig/ct_bearings_only_varying_sensors/setup/ipls.data};
\addlegendentry{\ipls}
\addplot+[semithick]
table [x, y]{fig/ct_bearings_only_varying_sensors/setup/lm-ipls.data};
\addlegendentry{\lmIpls}
\addplot+[semithick]
table [x, y]{fig/ct_bearings_only_varying_sensors/setup/ls-ipls.data};
\addlegendentry{\lsIpls}
\addplot [semithick, black, mark=*, mark size=0.25, mark options={solid}, only marks]
table [x, y]{fig/ct_bearings_only_varying_sensors/setup/true_traj.data};
\addlegendentry{$\xSeq$};
\end{axis}
\end{tikzpicture}
  \caption{
    \label{fig:ct_bearings_only_varying_sensors:setup}Single realisation of the CT experiment with varying bearings-only measurements.
    At $k = 50, 100, \dots, 500$ only a single low noise measurement is observed.
    For this particular realisation, it is only the \lm-regularised smoothers, \lmIeks and \lmIpls, which accurately estimate the general shape of the true trajectory.
}
\end{figure}
The importance of regularisation is shown in \cref{fig:ct_bearings_only_varying_sensors:metrics},
where \rmse and \nees metrics, averaged over \numMcRuns independent realisations, are displayed.
\begin{figure}
  \centering
  \begin{tikzpicture}[scale=1]
\pgfplotsset{/pgfplots/error bars/error bar style={very thick}}

\begin{groupplot}[
    group style={
        group size=2 by 1,
        horizontal sep=1.65cm,
    },
    xlabel={Iteration},
    height=5cm,
    width= 6.5cm,
    xmin=0.5, xmax=10.5,
    ylabel near ticks,
]
\nextgroupplot[
ymode=log,
legend cell align={left},
legend columns=6,
legend style={fill opacity=0.8, draw opacity=1, text opacity=1, at={(1,1.05)}, anchor=south, draw=white!80!black},
tick align=outside,
tick pos=left,
ylabel={\rmse},
x grid style={white!69.0196078431373!black},
xtick style={color=black},
y grid style={white!69.0196078431373!black},
ymin=0.2, ymax=1000,
ytick style={color=black},
cycle list name=onlyColor
]
\addplot+[error bars/.cd, y dir=both, y explicit]
table[x=x, y=y, y error=err] {fig/ct_bearings_only_varying_sensors/metrics/rmse/ieks.data};
\addlegendentry{\ieks}
\addplot+[error bars/.cd, y dir=both, y explicit]
table[x=x, y=y, y error=err] {fig/ct_bearings_only_varying_sensors/metrics/rmse/lm-ieks.data};
\addlegendentry{\lmIeks}
\addplot+[error bars/.cd, y dir=both, y explicit]
table[x=x, y=y, y error=err] {fig/ct_bearings_only_varying_sensors/metrics/rmse/ls-ieks.data};
\addlegendentry{\lsIeks}
\addplot+[error bars/.cd, y dir=both, y explicit]
table[x=x, y=y, y error=err] {fig/ct_bearings_only_varying_sensors/metrics/rmse/ipls.data};
\addlegendentry{\ipls}
\addplot+[error bars/.cd, y dir=both, y explicit]
table[x=x, y=y, y error=err] {fig/ct_bearings_only_varying_sensors/metrics/rmse/lm-ipls.data};
\addlegendentry{\lmIpls}
\addplot+[error bars/.cd, y dir=both, y explicit]
table[x=x, y=y, y error=err] {fig/ct_bearings_only_varying_sensors/metrics/rmse/ls-ipls.data};
\addlegendentry{\lsIpls}
\nextgroupplot[
ymode=log,
tick align=outside,
tick pos=left,
ylabel={\nees},
x grid style={white!69.0196078431373!black},
xtick style={color=black},
y grid style={white!69.0196078431373!black},
ymin=12, ymax=5e6,
ytick style={color=black},
cycle list name=onlyColor
]
\addplot+[error bars/.cd, y dir=both, y explicit]
table[x=x, y=y, y error=err] {fig/ct_bearings_only_varying_sensors/metrics/nees/ieks.data};
\addplot+[error bars/.cd, y dir=both, y explicit]
table[x=x, y=y, y error=err] {fig/ct_bearings_only_varying_sensors/metrics/nees/lm-ieks.data};
\addplot+[error bars/.cd, y dir=both, y explicit]
table[x=x, y=y, y error=err] {fig/ct_bearings_only_varying_sensors/metrics/nees/ls-ieks.data};
\addplot+[error bars/.cd, y dir=both, y explicit]
table[x=x, y=y, y error=err] {fig/ct_bearings_only_varying_sensors/metrics/nees/ipls.data};
\addplot+[error bars/.cd, y dir=both, y explicit]
table[x=x, y=y, y error=err] {fig/ct_bearings_only_varying_sensors/metrics/nees/lm-ipls.data};
\addplot+[error bars/.cd, y dir=both, y explicit]
table[x=x, y=y, y error=err] {fig/ct_bearings_only_varying_sensors/metrics/nees/ls-ipls.data};
\end{groupplot}
\end{tikzpicture}
  \caption{\label{fig:ct_bearings_only_varying_sensors:metrics}Simulated coordinated turn model with bearings only measurements. The plot shows averaged \rmse and \nees across iterations.
  Error bars correspond to the standard error, that is, the estimated standard deviation scaled by $1 / \sqrt{\numMcRuns}$.
}
\end{figure}
The \nees is sensitive to divergent trials but in terms of \rmse, it is clear that the regularised smoothers outperform the standard versions.
%
\section{Discussion and conclusions}
In this paper, we present extensions of the \ieks and \ipls smoothers in the form of \lm-regularised smoother \lmIpls and line-search smoothers \lsIeks, \lsIpls.
We build on existing work connecting the \ieks to \gn optimisation and derive a similar interpretation for the \ipls.
We show that \lm-regularisation can be achieved with a simple modification of the state-space model in the form of an added pseudo-measurement of the state.

We present simulation results that show that the proposed smoothers improve state-of-the-art smoothers in highly nonlinear settings, with an increase in computational burden,
with respect to standard iterated smoothers.
%
%
%
\bibliographystyle{unsrt}
\bibliography{refs}
\clearpage
\newpage
\appendix
\section{Theoretical derivations}
\subsection{Proof of gradient of the IPLS cost function }
\label{app:gn_ipls:grad}
Here, we provide the proof of \cref{thm:gn_ipls:cost_fn:grad}.
The gradient of the \ipls cost function in \cref{eq:ipls:gn:cost_fn} is
\begin{align*}
  \grad \gnIplsObj{i}(\xSeq) &= \frac{1}{2} \twoNorm{\gnRIt{i}(\xSeq)}^2
                             = J_{\gnRIt{i}}(\xSeq)^{\top} \gnRIt{i}(\xSeq).
\end{align*}

To compute $J_{\gnRIt{i}}(\xSeq)$, where $\gnRIt{i}$ is defined in \cref{eq:def:ipls:gnr},
we need the Jacobians of the \slr expectations $\stateBar(\cdot), \measBar(\cdot)$
(for brevity we use the shorthand $p_{\ts}(\x) = \normal(\x; \x_{\ts}, \PHatIt{i}_{\ts})$):
\begin{small}
\begin{align}
  \label{eq:ipls:proc_jacobian}
  J_{\stateBar_{\ts}}(\x_{\ts}) &= J(\expect{\procFn(\x_{\ts})})
                                = J \left( \int \procFn(\x) p_{\ts}(\x) \d \x \right) \nonumber \\
                                &= \int \procFn(\x) J(p_{\ts}(\x)) \d \x
                                = \int \procFn(\x) (\x - \x_{\ts})^{\top} p_{\ts}(\x) \d \x \PHatItInv{i}
\end{align}
\end{small}
Here, we use equation (S-9 in \cite{damped_plf}) and note that
\begin{small}
\begin{align}
  \label{eq:ipls:zero_int}
  \int \stateBar_{\ts}(\x_{\ts}) \left( \x - \x_{\ts} \right)^{\top} p_{\ts}(\x) \d \x
       &= \stateBar_{\ts}(\x_{\ts}) \left( \int \x^{\top}  p_{\ts}(\x) \d \x - \x_{\ts}^{\top}  \int p_{\ts}(\x) \d \x \right)
       = 0.
\end{align}
\end{small}
We substitute \cref{eq:ipls:zero_int} into the derivative in \cref{eq:ipls:proc_jacobian} to obtain
\begin{align}
  \label{eq:ipls:proc_bar:deriv}
  J_{\stateBar_{\ts}}(\x_{\ts})
  &= \int \left( \procFn(\x) - \stateBar_{\ts}(\x_{\ts}) \right) \left( \x - \x_{\ts} \right)^{\top} p_{\ts}(\x) \d \x\ \PHatItInv{i}\nonumber \\
  &= \PsiSlr{\procFn_{\ts}}^{\top}(\x_{\ts}) \PHatItInv{i} = \procAIt{i}_{\ts}(\x_{\ts}),
\end{align}
i.e., the \slr Jacobian in \cref{eq:slr:jacobian}, based on the estimates at iteration $i$.
Analogously, $J_{\measBar_{\ts}}(\x_{\ts}) = \measHIt{i}_{\ts}(\x_{\ts})$.

Given the definition of $\gnRIt{i}$ in \cref{eq:def:ipls:gnr}, let $\allDim = \tsFin \dimState + \tsFin \dimMeas$ and compute the Jacobian $J_{\gnRIt{i}}(\xSeq): \reals^{\tsFin \dimState} \to \reals^{\allDim \times \tsFin \dimState}$, in terms of the above Jacobians:
\begin{align*}
  J_{\gnRIt{i}}(\xSeq) &= J_{\augQIplsInvSqrtT \augObsIpls}(\xSeq)
  = \augQIplsInvSqrtT
  \begin{pmatrix}
    \procAIt{i}(\xSeq) &
    \measHIt{i}(\xSeq)
  \end{pmatrix}^\top
\end{align*}
where,
\begin{align}
  \label{eq:def:ipls:proc_jacobian_seq}
  \procAIt{i}(\xSeq) \coloneqq
  &\begin{pmatrix}
    I_{\dimState} &0 &\hdots &~ &0 \\
    -\procAIt{i}_1(\x_1) &I_{\dimState} &0 &\hdots &0 \\
    0 &-\procAIt{i}_2(\x_2) &I_{\dimState} &\hdots &0 \\
    \vdots &~ &\ddots &~ &\vdots \\
    0 &~ &\dots &I_{\dimState} &0 \\
    0 &~ &\dots & -\procAIt{i}_{\tsFin-1}(\x_{\tsFin-1}) &I_{\dimState}\\
  \end{pmatrix} \nonumber \\
\end{align}
and
\begin{align}
  \label{eq:def:ipls:meas_jacobian_seq}
  \measHIt{i}(\xSeq) \coloneqq
  &- \diag \left( \measHIt{i}_1(\x_1), \measHIt{i}_2(\x_2), \dots, \measHIt{i}It{i}_{\tsFin}(\x_{\ts}) \right).
\end{align}

\subsection{Proof details of the IPLS GN connection}
\label{app:gn_ipls}
The proof of \cref{thm:gn_ipls} has a similar structure as the proof of \cref{thm:gn_ieks},
which is outlined in \cref{fig:props_overview}.
The proof requires linearising and construction of the approximative GN cost function $\gnRLin$
and \slr linearisation of the state-space model in \cref{eq:def:state_space_model}.
We then show that this approximate GN objective corresponds to this approximate state-space model.
Here, we provide details of these steps.

The important part of the linearisation comes from \cref{thm:gn_ipls:cost_fn:grad}.
From the Jacobian $J_{\gnRIt{it}}$ we make the \gn approximation,
linearising around the current smoothed state sequence estimate $\xHatItSeq{i}$:
\begin{align}
  \label{eq:app:ipls:gnr:taylor}
  \gnRIt{i}(\xSeq) &\approx \gnRLin(\xSeq)
  = \gnRIt{i}(\xHatItSeq{i}) + J_{\gnRIt{i}}(\xHatItSeq{i}) (\xSeq - \xHatItSeq{i}) \nonumber \\
  &= \augQIplsInvSqrtT \augObsIpls(\xHatItSeq{i})
  + \augQIplsInvSqrtT
  \begin{pmatrix}
    \procAIt{i}(\xHatItSeq{i}) \\
    \measHIt{i}(\xHatItSeq{i})
  \end{pmatrix}
  (\xSeq - \xHatItSeq{i}) \nonumber \\
  &= \augQIplsInvSqrtT
  \underbrace{
  \begin{pmatrix}
    \x_1 - \priorMean \\
    \x_2 - \left( \procAIt{i}_{1}(\xHatIt{i}_1) \x_1 + \procB_{1}(\xHatIt{i}_1) \right ) \\
    \x_3 - \left( \procAIt{i}_{2}(\xHatIt{i}_2) \x_2 + \procB_{2}(\xHatIt{i}_2) \right ) \\
    \vdots \\
    \x_{\tsFin} - \left( \procAIt{i}_{\tsFin-1}(\xHatIt{i}_{\tsFin}) \x_{\tsFin} + \procB_{\tsFin-1}(\xHatIt{i}_{\tsFin}) \right ) \\
    \meas_1 - \left( \measHIt{i}_{1}(\xHatIt{i}_1) \x_1 + \measC_{1}(\xHatIt{i}_1) \right) \\
    \vdots \\
    \meas_{\tsFin} - \left( \measHIt{i}_{\tsFin}(\xHatIt{i}_{\tsFin}) \x_{\tsFin} + \measC_{\tsFin}(\xHatIt{i}_{\tsFin}) \right) \\
  \end{pmatrix}}_{\coloneqq \augObsIplsApprox(\xSeq)}
\end{align}
where, in the last equality, we use the relation in \cref{eq:slr:offset} to express $\gnRLin(\xSeq)$ in terms of the offsets $\procB_{1:\tsFin}(\cdot)$ and $\measC_{1:\tsFin}(\cdot)$.
\subsection{Proof of \cref{thm:lm_rts_smoothing}}
Here we detail the three steps of the proof of \cref{thm:lm_rts_smoothing}.
In step 1 we derive the approximate \lm objective.
From \cref{eq:lm:cost_fn} we have that the \lm objective $\lmObj{i}(\xSeq)$ is the sum of the \gn objective $\gnObj(\xSeq)$ and a regularisation term.
By simply extending $\gnR(\xSeq)$ in \cref{eq:def:ipls:gnr}, we can construct
\begin{align}
  \label{eq:def:ipls:lm:gnr}
  \gnRLM{i}(\xSeq) &=
  \begin{pmatrix}
    \gnR(\xSeq) &
    \lmRegSqrtT (\xSeq - \xHatItSeq{i})
  \end{pmatrix}^\top,
\end{align}
such that $\lmObj{i}(\xSeq) = \frac{1}{2}\twoNorm{\gnRLM{i}}^2$.

Since the regularisation term is linear in $\xSeq$, the result is
\begin{align}
  \gnRLM{i}(\xSeq) \approx
  \begin{pmatrix}
    \gnRLin(\xSeq) &
    \lmRegSqrtT (\xSeq - \xHatItSeq{i})
  \end{pmatrix}^\top,
\end{align}
where we note that the approximate objective is indeed the approximate \gn objective in \cref{eq:ipls:approx_objective} with added \lm regularisation
\begin{align}
  \label{eq:ipls:lm:approx_cost}
  \lmObjLin(\xSeq)
  &= \frac{1}{2}
  \twoNorm{
  \begin{pmatrix}
    \gnRLin(\xSeq) &
    \lmRegSqrtT (\xSeq - \xHatItSeq{i})
  \end{pmatrix}^\top
  }^2
  \nonumber \\
                   &= \gnObjLin(\xSeq)
                   + \frac{1}{2} \lmRegParamIt{i} (\xSeq - \xHatItSeq{i})^{\top} \left[ \regMat^{(i)}\right]^{-1} (\xSeq - \xHatItSeq{i}).
\end{align}

In step 2 we derive the negative log-posterior for the linearised state-space model with the measurement $\xHatIt{i}_{\ts}$ in \cref{eq:lm:extended_problem}.
The additional measurement will, for each timestep, contribute with a term $\frac{1}{2} \lmRegParamIt{i} (\x_{\ts} - \xHatIt{i}_{\ts})^{\top} \left[ \regMat^{(i)}_{\ts}\right]^{-1} (\x_{\ts} - \xHatIt{i}_{\ts})$ to the negative log-posterior.
All the extra measurements can be combined into a single term $\frac{1}{2} \lmRegParamIt{i} (\xSeq - \xHatItSeq{i})^{\top} \left[ \regMat^{(i)}\right]^{-1} (\xSeq - \xHatItSeq{i})$,
with the block-diagonal matrix $\regMat^{(i)} = \diag(\regMat_1^{(i)}, \regMat_2^{(i)}, \dots, \regMat_{\tsFin}^{(i)}),\, \regMat^{(i)}_{\ts} \in \reals^{\dimState \times \dimState}$.
Note that we restrict ourselves to $\regMat^{(i)}$ on this form, whereas other suitable options exist \cite{num_opt,nonlinear_regression}.
The measurement model for $\xHatItSeq{i}$ is linear in $\xSeq$ so the negative log-posterior of the approximated state-space model produced by \ieks or \ipls linearisation will be extended with this term.

In step 3 we compare the linearisations.
We know from \cref{thm:gn_ieks,thm:gn_ipls} that the log-posterior without the measurement is $\gnObjLin(\xSeq)$ and after introducing the measurement the log-posterior (up to a constant)
is therefore $\lmObjLin(\xSeq)$ in \cref{eq:ipls:lm:approx_cost}.
We conclude that the algorithms yield the same result since they minimise the same loss function in closed form.
It follows that an iteration of the \ieks or \ipls for this extended state-space model is equivalent to an iteration of \lm optimisation of the corresponding cost function.
\subsection{Armijo and Wolfe step length conditions}
This section provides the Armijo and Wolfe step length conditions for the \lsIpls method, see \cref{eq:def:armijo_cond,eq:def:wolfe_cond}.
An inexact line-search method only requires a sufficient decrease in the cost function, rather than finding the minimum along the search direction $\searchDirSeq{i}$.
A sufficient decrease is guaranteed by fulfilling two conditions of the cost function and its gradient at a point on the line.

The first condition is commonly referred to as the Armijo condition,
and combined they are called the Wolfe conditions or the Armijo--Wolfe conditions\cite{num_opt}.
The Armijo condition is a sufficient condition for a decreasing cost function,
whereas the Wolfe condition ensures that the step length $\alpha$ is large enough to give faster convergence.

To check the Armijo--Wolfe conditions for a certain step length $\alpha$ we need to compute the gradient of the cost function,
which comes directly from \cref{thm:gn_ipls:cost_fn:grad}.

Due to the sparseness of the Jacobian, the directional derivative can be computed more efficiently by evaluating the product
\begin{align}
  &(\searchDirSeq{i})^{\top} \grad \costFn(\xHatItSeq{i})
  = (\searchDir{i}_{1})^{\top} \priorCov^{-1}(\xHatIt{i}_1 - \priorMean) \nonumber \\
  &\hspace{1cm}+ \sum_{\ts=1}^{K-1} \left( \searchDir{i}_{\ts+1} - \procA_{\ts}(\xHatIt{i}_\ts) \searchDir{i}_{\ts} \right)^{\top}
  \times (\procCov_{\ts} + \procLinCovIt{i}_{\ts})^{-1}(\xHatIt{i}_{\ts+1} - \stateBar_\ts(\xHatIt{i}_\ts)) \nonumber \\
  &\hspace{1cm}- \sum_{\ts=1}^{K} \left( \measH_{\ts}(\xHatIt{i}_\ts) \searchDir{i}_{\ts} \right)^{\top} (\measCov_{\ts} + \measLinCovIt{i}_{\ts})^{-1}(\xHatIt{i}_{\ts+1} - \measBar_\ts(\xHatIt{i}_\ts)).
\end{align}

The derivations for the gradient of the \ieks cost function are analogous,
but instead using the actual measurement and motion models and their respective analytical Jacobians.
\section{Experimental details}
\label{sec:app:exp_details}
%
For all simulations, we use the same coordinated turn (CT) motion model with state
$
    \mathbf{x} = \begin{pmatrix}
        x &
        y &
        \dot{x} &
        \dot{y} &
        \omega
    \end{pmatrix}^\top,
$
and a constant motion model for all $\ts$
\begin{equation}
    \procFn_\ts(\mathbf{x}) = 
    \begin{pmatrix}
    x + \frac{\sin(\omega t)}{\omega} \dot{x} - \frac{\cos(\omega t) - 1}{\omega} \dot{y} \\
    y + \frac{\cos(\omega t) -1 }{\omega} \dot{x} + \frac{\sin(\omega t)}{\omega} \dot{y} \\
    \cos(\omega t) \dot{x} + \sin(\omega t) \dot{y} \\
    - \sin(\omega t) \dot{x} + \cos(\omega t) \dot{y} \\
    \omega
    \end{pmatrix},
\end{equation}
where $t=0.01$.
The covariance matrix of the process noise (constant for all timesteps) is
\begin{equation}
    Q_{\ts} =
    \begin{pmatrix}
    \sigma_c t^3 / 3 & 0 & \sigma_c t^2 / 2 & 0 & 0 \\
    0 & \sigma_c t^3 / 3 & 0 & \sigma_c t^2 / 2 & 0 \\
    \sigma_c t^2 / 2 & 0 & \sigma_c t & 0 & 0 \\
    0 & \sigma_c t^2 / 2 & 0 & \sigma_c t & 0 \\
    0 & 0 & 0 & 0 & \sigma_\omega t,
    \end{pmatrix}
\end{equation}
with $\sigma_c = 0.01$ and $\sigma_\omega = 10$.

True trajectories with $\tsFin = 500$ timesteps are sampled from this motion model with prior mean 
$\priorMean = (0, 0, 1, 0, 0)^{\top}$ and covariance $\priorCov = \diag(0.1, 0.1, 1, 1, 1)$.

\end{document}